\documentclass{amsart}
\usepackage[T1]{fontenc}
\usepackage{amsmath, amsthm, amsfonts, amssymb}
\usepackage{mathrsfs}
\usepackage{graphicx}
\usepackage{url}
\usepackage{enumerate}
\usepackage{fancyhdr}
\usepackage{xspace}
\usepackage{verbatim}
\usepackage{dsfont}

\usepackage{tikz,xcolor}
\usetikzlibrary{arrows,decorations.pathmorphing,backgrounds,fit,automata,shapes.geometric,shapes.arrows,calc,positioning}
\theoremstyle{definition}

\newtheorem{defn}{Definition}[section]

\theoremstyle{remark}
\newtheorem{rmk}[defn]{Remark}

\theoremstyle{plain}

\newtheorem{cor}[defn]{Corollary}

\newtheorem{lem}[defn]{Lemma}

\newtheorem{thm}[defn]{Theorem}

\numberwithin{equation}{section}

\newcommand{\SG}{\text{SG}} 
\newcommand{\DF}{\mathcal{E}}
\newcommand{\domDF}{\mathcal{F}}
\newcommand{\Hil}{\mathcal{H}}
\newcommand{\Mag}{\mathcal{M}}

\DeclareMathOperator{\dom}{dom}
\DeclareMathOperator{\Tr}{Tr}

\def\s-s{self-similar}

\DeclareMathSizes{10}{9}{7}{5}

\def\Tri#1#2#3#4{{ 
\put(#1,#2){\line(3,5){#3}} \put(#1,#2){\line(3,0){#4}} \count223=#1 \advance\count223 by #4
\put(\count223,#2){\line(-3,5){#3}}}}

\def\Spic#1#2#3#4{{ 
\count205=#2\count206=#2\count207=#2\count208=#2\count209=#2
\count210=#1\count211=#3\count214=#3\count212=#4 \divide\count210 by 2 {\ifnum\count210>0
\Spic{\count210}{#2}{#3}{#4} \multiply\count207 by 3 \multiply\count208 by 6 \multiply\count209 by
5 \multiply\count207 by \count210 \multiply\count208 by \count210 \multiply\count209 by \count210
{\advance\count211 by \count207 \advance\count212 by \count209
\Spic{\count210}{#2}{\count211}{\count212}} {\advance\count214 by \count208
\Spic{\count210}{#2}{\count214}{#4}} \else \multiply\count205 by 3 \multiply\count206 by 6
\Tri{#3}{#4}{\count205}{\count206} \fi }}}

\def\Separatedfirstlevel#1{{ 
\count301 =#1 \multiply \count301 by 5 \divide\count301 by 3 
\count302=#1 \multiply\count302 by 3 
\count303=#1  \multiply\count303 by 6 
\count304=\count302 \advance\count304 by #1 
\count305=#1 \multiply \count305 by 5 
\count306=\count305 \advance\count306 by \count301 
\Tri{0}{0}{\count302}{\count303}  
\Tri{\count304}{\count306}{\count302}{\count303} 
\multiply\count304 by 2
\Tri{\count304}{0}{\count302}{\count303}
}}

\allowdisplaybreaks[4]
\title[Magnetic Laplacian on SG]{Magnetic Laplacians of locally exact forms on the Sierpinski Gasket}

\author[Hyde, Kelleher, Moeller, Rogers, Seda]{Jessica~Hyde,  Daniel~Kelleher, Jesse~Moeller, Luke~G.~Rogers, Luis~Seda}

\thanks{Authors supported in part by the National Science Foundation through grant DMS-0505622.}

\begin{document}
%

\begin{abstract}
We give a mathematically rigorous construction of a magnetic Schr\"{o}dinger operator corresponding to a field with flux through finitely many holes of the Sierpinski Gasket.  The operator is shown to have discrete spectrum accumulating at $\infty$, and it is shown that the asymptotic distribution of eigenvalues is the same as that for the Laplacian. Most eigenfunctions may be computed using gauge transformations corresponding to the magnetic field and the remainder of the spectrum may be approximated to arbitrary precision by using a sequence of approximations by magnetic operators on finite graphs.

\end{abstract}
%

\maketitle
%

\tableofcontents
%

\section{Introduction} \label{introduction}

The properties of  quantum electron on a fractal substrate and under the influence of a magnetic field were studied long ago in the physics literature~\cite{Alexanderetal,Alexander,Rammal,AO,Rammal2,Ghez} as part of a more general program involving quasiperiodic media~\cite{Simon,Bel}, but until recently there has been no mathematically rigorous model for even formulating a magnetic Schr\"{o}dinger equation on a self-similar fractal set.  We remedy this in the special case of the Sierpinski Gasket with certain simple magnetic fields using  mathematical developments from the study of diffusions and  Laplacian-type operators on fractals using probability and functional analysis (see~\cite{BarlowPerkins,Kigamibook,Strichartzbook} and references therein) and the recent  introduction of differential forms associated to this structure~\cite{CS,IRT,ACSY,CSetal,HT,hinzetal,HR}.   These developments in analysis on fractals have benefited from and contributed to the understanding of quantum and statistical physics~\cite{ADT09,ADT10,ABDTV12,Dunne12,Akk}.

Our goal in this paper is to introduce a mathematically rigorous Schr\"{o}dinger equation for a magnetic operator on the Sierpinski Gasket (SG), following the methods of~\cite{IRT,HT,hinzetal,HR}, and study its spectrum, which by~\cite{HR} is  discrete and accumulates only at $\infty$ (Theorem~\ref{thm:mainresultofHR}).   For reasons of mathematical simplicity we consider a somewhat unphysical situation in which the magnetic field has non-zero flux through only finitely many of the ``holes'' in the gasket.  In this situation we are able to prove that the magnetic operator may be approximated in an appropriate sense by a renormalized sequence of magnetic operators on approximating graphs (Theorems~\ref{thm:DFamconvergestoDFa} and~\ref{thm:convergeofMagm}).  This approximation generalizes the well-known approximation of a Dirichlet form on SG by renormalized graph Dirichlet forms~\cite{Kigami89,Kigamibook}.  The  approximating magnetic operators provide a method for numerical study of the spectrum and  some data of this type is in Section~\ref{sec:spectra}. Guided by the observations in this data and using the description of the Laplacian spectrum from the spectral decimation method~\cite{RammalToulouse,FukushimaShima,MT} we show that a field through only finitely many holes of SG modifies only those eigenvalues for which the eigenfunctions have support enclosing these holes (Theorems~\ref{thm:conjugateLapefns} and~\ref{thm:spectrumofMaga}), and conclude that the spectral asymptotics of the magnetic operator are the same as those of the Laplacian (Corollary~\ref{cor:spectasymptotics}).  In principle, for any magnetic field of this type,  one can use our methods to compute the bulk of the spectrum and the associated eigenfunctions by applying suitable gauge transformations to Laplacian eigenfunctions. For the small (asymptotically vanishing) portion of the spectrum that is not found by this method one can choose $\lambda$ and compute all eigenvalues of size less than $\lambda$ by solving finitely many linear algebra problems.  We also give a description of the basic modification that a magnetic field makes to the Laplacian spectrum by examining periodic functions on a covering space (Section~\ref{sec:ladderfractafold}).  In the case of SG the relevant covering space is a fractafold (as defined in~\cite{STR03}) called the Sierpinski Ladder~\cite{ST12}.

\section{Analysis and $1$-forms on SG}\label{sec:SG}

The Sierpinski Gasket (SG) is the attractor of the Iterated Function System $\{ F_j = \frac{1}{2} (x - p_j ) + p_j \}$, $j=0,1,2$ for $\{p_{j}\}$ the vertices of an equilateral triangle in $\mathbb{R}^2$. The image of SG under an $m$-fold composition of these maps is called an $m$-cell.  We index these by words: let $w=w_{1}w_{2}\dotsc w_{m}\in\{0,1,2\}^{n}$ be a word of length $|w|=m$ and $F_{w}=F_{w_{1}}\circ \dotsm \circ F_{w_{m}}$.  Then $F_{w}(\SG)$ is an $m$-cell.  From the cellular structure of SG we obtain a sequence of graphs.  Let $V_{0}=\{p_{0},p_{1},p_{2}\}$ and inductively $V_{m}=\cup_{j=0,1,2} F_{j}(V_{m-1})$.  The $m^{\text{th}}$-level graph approximation of SG is the graph with vertices $V_{m}$ and edges between pairs of vertices that are contained in a common $m$-cell. We write $x\sim_{m}y$ to denote that there is an edge between $x,y\in V_{m}$ in the $m$-scale graph.  The set $V_{\ast}=\cup_{m}V_{m}$ is dense in SG.

Analysis on SG is based on the existence of a Dirichlet form and an associated Laplacian. Of the available constructions~\cite{BarlowPerkins,Lindstrom,Kigamibook} we follow the method of Kigami~\cite{Kigamibook}, some features of which are as follows. Proofs of all of the results stated may be found in~\cite{Kigamibook,Strichartzbook}. We endow SG with the (unique)  self-similar probability measure $\mu$ that is invariant under the symmetries of the triangle with vertices the points $p_j$.

\begin{enumerate}[{A}1]
\item\label{AselfsimilardecompofDF} There is a Dirichlet form $\DF$ on SG with domain $\domDF\subset L^{2}(\mu)$ consisting of continuous functions. $\DF$ may be localized to any $m$-cell and is self-similar with scaling factor $\frac{5}{3}$.  Specifically, for a word $w$ with $|w|=m$ let $\DF_{w}(f,g)=(\frac{5}{3}\bigr)^{m}\DF(f\circ F_{w},g\circ F_{w})$ so $\DF_{w}$ is a Dirichlet form on $F_{w}(\SG)$.  Then $\DF(f,g)=\sum_{|w|=m}\DF_{w}(f,g)$.
\item\label{ADFislimitofgraphs} $\DF$ may be obtained as a limit of forms on the graphs.  For $f,g: V_{*} \to \mathbb{R}$, $m \in \mathbb{N}$, let $\DF_{m}(f,g)=\sum_{x\sim_{m}y} (f(x)-f(y))(g(x)-g(y))$.  Then $\bigl(\frac{5}{3}\bigr)^{m}\DF_{m}(f,f)$ is non-decreasing and converges to $\DF(f,f)$ if $f\in\domDF$.
\item From standard considerations there is a non-positive definite self-adjoint Laplacian associated to $\DF$. We define $f\in\dom(\Delta)\subsetneq\domDF$ to mean there is a continuous $\Delta f$ such that  $\DF(f,g)=\langle -\Delta f,g\rangle_{L^{2}}$ for all $g\in\domDF_{0}$, where $\domDF_{0}\subset\domDF$ is the subspace of functions that vanish on $V_{0}$. 
\item Let $df(p_{j}) = \lim_{m\to\infty} \bigl(\frac{5}{3}\bigr)^{m} \bigl(2f(p_{j}) -f(F_{j}^{\circ m} p_{(j+1)})-f(F_{j}^{\circ m} p_{(j+2)})\bigr)$, where the subscripts are taken modulo $3$.  If $f\in\dom(\Delta)$ then this limit exists on $V_{0}$ and there is a Gauss-Green formula $\DF(f,g)=\langle -\Delta f,g\rangle_{L^{2}}+\sum_{j=0}^{2} df(p_{j})g(p_{j})$;  we call $df(p_{j})$ the normal derivative of $f$ at $p_{j}$.  Both $df$ and the Gauss-Green formula may be localized to any $m$-cell.
\item\label{Delta_is_graph_limit} $\Delta$ may be obtained as a limit of graph Laplacians. For $f : V_{*} \to \mathbb{R}$, $m \in \mathbb{N}$, let $\Delta_m f(x) = \sum_{y \sim_{m} x} (f(y) - f(x))$.  Then $\Delta f(x) = \frac{3}{2} \lim_{m\to\infty} 5^{m}\Delta_{m}f(x)$
\item\label{Aharmonic} If $X\subset\SG$ is finite and $g:X\to\mathbb{R}$ then there is a unique $f\in\domDF$ such that $f|_{X}=g$ and $\DF(f)$ is minimized; $f$ is called the harmonic extension of $g$ and satisfies $\Delta f(x)=0$ for $x\in\SG\setminus X$.  If $X=V_{m}$ then also $\Delta_{n}f(x)=0$ for all $n>m$ and $x\in V_{n}\setminus V_{m}$, and $f$ is called $m$-harmonic.
\end{enumerate}

Differential forms on certain spaces that include the Sierpinski Gasket have been studied in~\cite{CS,IRT,ACSY,CSetal,HT,hinzetal}.  We follow the approach in~\cite{IRT}, which introduces $1$-forms as  a Hilbert space $\Hil$ generated by tensor products $f\otimes g$ with $f,g\in\domDF$, and which is a module over $\domDF$.  There is then a derivation $\partial:\domDF\to\Hil$ such that $\|\partial f\|_{\Hil}^{2}=\DF(f,f)$ and the image of $\partial$ is the space of exact forms.

The key feature that we need from~\cite{IRT} is that the action of $\domDF$ on $\Hil$ by multiplication extends to permit multiplication by much more general functions.  In particular, multiplication by the characteristic function $\mathds{1}_{w}$ of an $m$-cell $F_{w}(X)$ is well-defined.  This permits a cellular decomposition of $\Hil$ akin to that described in~(A\ref{AselfsimilardecompofDF}) and a notion of graph approximation like that in~(A\ref{ADFislimitofgraphs}).  Proofs of the following results are in~\cite{IRT}.
\begin{enumerate}[{F}1]
\item \label{Fselfsimilar} Let $\Hil_{w}$ be the space of $1$-forms constructed from $\bigl(\DF_{w}, \domDF|_{F_{w}(\SG)}\bigr)$ in the same manner as $\Hil$ is constructed from $(\DF,\domDF)$. If $h_{w}=f|_{F_{w}(\SG)}\otimes g|_{F_{w}(\SG)}$ then the map $h_{w}\mapsto (f\circ F_{w})\otimes (g\circ F_{w})=h$ takes the dense subspace of generators of $\Hil_{w}$ to those of $\Hil$ and has $\|h_{w}\|_{\Hil_{w}}^{2}=\bigl(\frac{5}{3}\bigr)^{m} \|h\circ F_{w}\|_{\Hil}^{2}$, so extends to an isomorphism of $\Hil_{w}$ to $\Hil$.
\item\label{Fcelldecomp} $\Hil_{w}$ is isometrically isomorphic to the subspace $\{a\mathds{1}_{w}:a\in\Hil\}$ via the continuous extension of the identification of $f|_{F_{w}(\SG)}\otimes g|_{F_{w}(\SG)}$  with $(f\otimes g)\mathds{1}_{w}$ and there is a direct sum decomposition $\Hil=\bigoplus_{|w|=m}\Hil_{w}$.
\item\label{FspacesHilm} Let $\Hil_{m}$ be the subspace of $\Hil$ generated by $\bigl\{f\otimes\mathds{1}_{w}: f\text{ is $m$-harmonic and } |w|=m \bigr\}$.  Then $\Hil_{m}\subset\Hil_{m+1}$ for all $m$ and $\cup_{m}\Hil_{m}$ is dense in $\Hil$.  The preceding results imply that $\Hil_{m}$ is isomorphic to a direct sum of copies of $\Hil_{0}$, with  one copy of $\Hil_{0}$ for each $m$-cell.  Moreover $\Hil_{0}$ is isomorphic to the harmonic functions modulo constants on SG, and is obtained from this space by applying the derviation $\partial$.
\end{enumerate}
Though it is not made explicit in~\cite{IRT}, the result in~(F\ref{FspacesHilm})  gives a connection to $1$-forms on graphs.  Recall that a $1$-form on a graph is a simply a function on the set of directed edges.  Let $a\in\Hil_{m}$ and $e_{xy}$ denote the edge from $x$ to $y$ in the $m$-scale graph.  Take $w$ with $|w|=m$ so $F_{w}(\SG)$ is the unique cell containing $e_{xy}$ and use~(F\ref{FspacesHilm}) to obtain  a harmonic function modulo constants $A_{w}$ corresponding to $a\mathds{1}_{w}$.  If we set $A(e_{xy})=A_{w}(y)-A_{w}(x)$ then $A$ is a well-defined function on directed edges, so is a $1$-form on the $m$-scale graph.  Moreover it is exact at scale $m$ because on each $m$-cell $F_{w}(\SG)$ it is the derivative of $A_{w}$.  The norm of $a\in\Hil_{m}$ is simply $\|a\|_{\Hil}^{2}=\sum_{|w|=m} \DF(A_{w})=\sum_{x\sim_{m}y}A(e_{xy})^{2}$.

This permits us to understand the space $\Hil$ as a generalization of $(\DF,\domDF)$, because it exhibits the $\Hil$-norm as a renormalized limit of $L^{2}$-norms. To make this connection more precise we need some definitions.  Let $h_{j}$ denote the harmonic function on $\SG$ which has values $h_{j}(p_{j})=0$, $h_{j}(p_{j-1})=-1$ and $h_{j}(p_{j+1})=1$.
\begin{defn}\label{defn:trace}
 For any two points joined by an edge in the $m$-scale graph there is $j\in\{0,1,2\}$ and  a word $w$ with $|w|=m$  such that the points are $x=F_{w}(p_{j-1})$ and $y=F_{w}(p_{j+1})$ (subindices are taken modulo $3$).  Define $\Tr_{m}:\Hil\to\Hil_{m}$ by setting the value on the edge $e_{xy}$ from $x$ to $y$  to be
\begin{equation*}
	(\Tr_{m}a) (e_{xy}) = \frac{1}{3} \langle a, \partial h_{j} \rangle_{\Hil}.
	\end{equation*}
A sequence $\{a_{m}\}_{1}^{\infty}\subset\Hil$ is called {\em compatible} if $\Tr_{m}a_{m+1}=a_{m}$ for all $m$.
\end{defn}
%
%
The following theorem should be compared to the results in Section~4 of~\cite{ACSY}.  It gives a full description of $1$-forms on SG as limits of $1$-forms on the approximating graphs.
\begin{thm}
The map $\Tr_{m}:\Hil\to\Hil_{m}$ is a projection.  If $a\in\Hil$ then the sequence $\{a_{m}\}$ of projections onto $\Hil_{m}$ is compatible, $a_{m}\to a$ in $\Hil$ and $\|a_{m}\|_{\Hil}\uparrow\|a\|_{\Hil}$.  Conversely, if $\{a_{m}\}$ is a compatible sequence then $a_{m}\in\Hil_{m}$ for all $m$; if we further assume that $\|a_{m}\|_{\Hil}$ is bounded then there is $a\in\Hil$ such that $a_{m}\to a$ and $a_{m}$ is the projection of $a$ to $\Hil_{m}$ for all $m$.
\end{thm}
\begin{proof}
The main thing we need to prove is that $\Tr_{m}$ is the projection onto $\Hil_{m}$. From~(F\ref{Fcelldecomp}) it is apparent that the projection can be taken one cell at a time, and the self-similarity in~(F\ref{Fselfsimilar}) implies that all cells are the same, so it suffices to show $\Tr_{0}$ is the projection onto $\Hil_{0}$.  We recall that $\Hil_{0}$ is obtained from the $2$-dimensional space of harmonic functions on SG by applying the derivation.

Let $\tilde{h}_{j}$ be harmonic on SG with $\tilde{h}_{j}(p_{j})=1$, $\tilde{h}_{j}(p_{j+1})=\tilde{h}_{j}(p_{j-1})=0$.   Symmetry shows that $h_{j}$ and $\tilde{h}_{j}$ are orthogonal, so $\partial h_{j}$ and $\partial\tilde{h}_{j}$ are an orthogonal basis for $\Hil_{0}$.  Suppose we project $a\in\Hil$  onto $a_{0}\in\Hil_{0}$ and compute the corresponding function $A_{\emptyset}$.  Since $\tilde{h}_{j}(p_{j+1})=\tilde{h}_{j}(p_{j-1})$, the difference $A(p_{j+1})-A(p_{j})$  is determined by the component involving $h_{j}$.  Precisely, it is  
\begin{equation*}
	A(p_{j+1})-A(p_{j})
	= \frac{1}{\DF(h_{j})} \langle a,\partial h_{j}\rangle_{\Hil} \bigl(h(p_{j+1})-h(p_{j})\bigr)
	= \frac{2}{6} \langle a,\partial h_{j}\rangle_{\Hil}
	=\Tr_{0}a (e_{p_{j-1}p_{j}}).
	\end{equation*}
Thus the trace assigns the same values to the edges as does the projection, and they must coincide.  Note that, in particular, this means the values of $\Tr_{0}a$ on the three edges $e_{01}$, $e_{12}$, $e_{20}$ must sum to zero, and indeed we find from the definition that they do because $\sum_{j}h_{j}$ is identically zero, so $\sum_{j}\Tr_{0}a(e_{j(j+1)})=0$.

Having established that $\Tr_{m}$ is the projection onto $\Hil_{m}$ it is immediate that the sequence of projections $a_{m}$ of $a\in\Hil$ is compatible, and~(F\ref{FspacesHilm}) shows $a_{m}\to a$ in $\Hil$ and $\|a_{m}\|_{\Hil}\uparrow\|a\|_{\Hil}$.  For the converse, if $a_{m}$ is a compatible sequence then the fact that $a_{m}=\Tr_{m}a_{m+1}$ implies $a_{m}\in\Hil_{m}$ for all $m$ and that $\|a_{m}\|_{\Hil}$ is an increasing sequence.  If we suppose that  $\|a_{m}\|_{\Hil}$  is bounded then using the Pythagorean  decomposition $\|a_{n}\|_{\Hil}^{2}=\|a_{m}\|_{\Hil}^{2}+\|a_{n}-a_{m}\|_{\Hil}^{2}$, $n>m$, for projection in a Hilbert space we see $\|a_{n}-a_{m}\|_{\Hil}^{2}\leq \bigl(\sup_{n}\|a_{n}\|_{\Hil}^{2}\bigr)-\|a_{m}\|_{\Hil}^{2}\to0$ as $m,n\to\infty$, so the sequence is Cauchy with limit $a\in\Hil$.  Finally, the composition $\Tr_{m}\circ\Tr_{m+1}\circ\dotsm\circ\Tr_{n}$ shows $a_{m}=\Tr_{m}a_{n}$ for all $n>m$ and, by taking the limit, $a_{m}=\Tr_{m}a$.
\end{proof}


\section{Magnetic forms, Magnetic Laplacian and gauge transformations}\label{sec:magnetic}
Following~\cite{hinzetal, HR} a magnetic differential  may be defined as a deformation of $\partial$.  To do so we treat a real-valued $1$-form $a\in\Hil$ as an operator $\domDF\to\Hil$ via multiplication, so $f\mapsto fa$.  Then $(\partial+ia):\domDF\to\Hil$ is the magnetic differential obtained by deforming $\partial$ via the form $a\in\Hil$.  With this approach an essential result is the following theorem.
\begin{thm}[\protect{\cite{HR}}]\label{thm:mainresultofHR}
The quadratic form $\DF^{a}(f)=\|(\partial+ia)f\|_{\Hil}^{2}$ with domain $\domDF$ is closed on $L^{2}(\mu)$.  Thus there is an associated non-positive definite self-adjoint magnetic (Neumann) Laplacian $\Mag^{a}_{N}$ satisfying $$\DF^{a}(f,g)=\langle -\Mag^{a}_{N}f,g\rangle_{L^{2}(\mu)}$$ for all $g\in\domDF$. Moreover $\Mag^{a}_{N}$ has compact resolvent, so the spectrum of $-\Mag^{a}_{N}$ is a sequence $0\leq\kappa_{1}\leq\kappa_{2}\leq\dotsm$ accumulating only at $\infty$.
\end{thm}
The same argument provides that the quadratic form $(\DF^{a},\domDF_{0})$ is closed on the space $L^{2}(\SG\setminus V_{0},\mu)$ and defines a magnetic (Dirichlet) Laplacian $\Mag^{a}_{D}$ with compact resolvent and $\DF^{a}(f,g)=\langle-\Mag^{a}_{D}f,g\rangle$ for all $g\in\domDF_{0}$.   Henceforth we will just use the Dirichlet magnetic operator and will denote it $\Mag^{a}$.  Much of our work transfers to the Neumann magnetic operator with minor changes.

\begin{rmk}
We are using the complexification of each of the spaces $L^2(\mu)$, $\domDF$, $\Hil$, $\dom(\Delta)$ as well as the subspaces $\domDF_0$, $\Hil_m$, etc.  These are standard, but for the convenience of the reader we recall that one may complexify $\domDF$ by endowing $\domDF+i\domDF$ with the form
\begin{equation*}
	\DF(f,g) = \DF(f_1,g_1) -i\DF(f_1,g_2) + i\DF(f_2,g_1) +\DF(f_2,g_2)
	\end{equation*}
where $f=f_1+if_2$ and $g=g_1+ig_2$. In this case the finite approximations in~(A.\ref{ADFislimitofgraphs}) become $\DF_m(f,g)=\sum_{x\sim_m} (f(x)-f(y))(\overline{g(x)-g(y)})$.  One may then construct $\Hil$ from the complexified version of $\domDF\otimes\domDF$ in the same manner as was done in the real case in~\cite{IRT} and discussed in Section~\ref{sec:SG}.
\end{rmk}

We wish to study the spectrum of $\Mag^{a}$ by making graph approximations.  For this reason we introduce a graph magnetic form and a graph magnetic Laplacian.  The connection between these and $\DF^{a}$ and $M^{a}$  is not immediately obvious but will rapidly become apparent.

\begin{defn}
Suppose $a\in\Hil$ is real-valued and for each $m\in\mathbb{N}$ let $a_{m}$ be the projection of $a$ to $\Hil_{m}$.  For $f,g: V_{*} \to \mathbb{C}$ define
\begin{gather}
	\DF_{m}^{a_{m}}(f) =\sum_{x,y: x\sim_{m} y} \Bigl| f(x)- f(y)e^{i a_{m}(e_{xy})} \Bigr|^{2} \\
	\Mag^{a_{m}}_{m} f(x) = - \sum_{y:y\sim_{m}x} \Bigl( f(x)- f(y)e^{i a_{m}(e_{xy})} \Bigr) \quad\text{ for }x\in V_{m}\setminus V_{0}.
	\end{gather} 
\end{defn}
We have the usual relation
\begin{equation} \label{eqn:DFmandMagfm}
	\DF_{m}^{a_{m}}(f,g) = \langle -\Mag^{a_{m}}_{m} f, g \rangle_{l^{2}(V_{m})},
\end{equation}
when $g=0$ on $V_0$, as may be verified by direct computation:
\begin{align*}
	\lefteqn{2 \sum_{x,y: x\sim_{m} y} \Bigl( f(x)- f(y)e^{i a_{m}(e_{xy})} \Bigr)\overline{\Bigl( g(x)- g(y)e^{i a_{m}(e_{xy})} \Bigr)}  } \quad&\notag\\
	&= \sum_{x\in V_{m}\setminus V_0} \overline{g(x)} \sum_{y\sim_{m}x} \Bigl( f(x)- f(y)e^{i a_{m}(e_{xy})} \Bigr) 
		- \sum_{y\in  V_{m}\setminus V_0} \overline{g(y)} \sum_{ x\sim_{m}y} \Bigl( f(x) e^{-i a_{m}(e_{xy})} - f(y) \Bigr) \notag\\
	&= \sum_{x\in V_{m}\setminus V_0} \overline{g(x)} \sum_{y\sim_{m}x} \Bigl( f(x)- f(y)e^{i a_{m}(e_{xy})} \Bigr) 
		+ \sum_{x\in  V_{m}\setminus V_0} \overline{g(x)} \sum_{ y\sim_{m}x} \Bigl( f(x) - f(y)e^{i a_{m}(e_{yx})} \Bigr) \notag\\
	&=2\sum_{x\in  V_{m}\setminus V_0} \bigl( -\Mag_{m}^{a_{m}}f(x) \bigr) \overline{g(x)}.
	\end{align*}
Note that we need not sum over $V_0$ because $g$ vanishes there. The equality holds for arbitrary $g$ if $ \sum_{y\sim_{m}x} \Bigl( f(x)- f(y)e^{i a_{m}(e_{xy})} \Bigr) =0$ for $x\in V_0$.

\begin{lem}
$\Bigl(\frac{5}{3}\Bigr)^{m}\DF_{m}^{a_{m}}(f)$ converges as $m\to\infty$ if and only if $f\in\domDF$.
\end{lem}
\begin{proof}
Observe from $\bigl| |f(x)|-|f(y)|\bigr|\leq \bigl| f(x)-f(y)e^{i a_{m}(e_{xy})} \bigr|$ that the convergence in the statement implies $\Bigl(\frac{5}{3}\Bigr)^{m}\DF_{m}(|f|)$ is finite and therefore $|f|\in\domDF$.  In particular $f$ is bounded.  The converse assumption $f\in\domDF$ also ensures $f$ is bounded.

Using boundedness of $f$ we may estimate as follows
\begin{align*}
	\Bigl| f(x) - f(y)e^{i a_{m}(e_{xy})} \Bigr|^{2} 
	&\leq \Bigl( \bigl| f(x) - f(y) \bigr| + |f(y)| \bigl| 1-e^{ia_{m}(e_{xy})} \bigr| \Bigr)^{2}\\
	&\leq 2 | f(x) - f(y) |^{2} + 2\|f\|_{\infty}^{2} |a_{m}(e_{xy})|^{2} 
	\end{align*}
and similarly
\begin{equation*}
	| f(x) - f(y) |^{2}
	\leq \Bigl| f(x) - f(y)e^{i a_{m}(e_{xy})} \Bigr|^{2}  + 2 \|f\|_{\infty}^{2} |a_{m}(e_{xy})|^{2}.
	\end{equation*}
As previously discussed, $\Bigl(\frac{5}{3}\Bigr)^{m}\sum_{x\sim_{m}y}|a_{m}(e_{xy})|^{2}=\|a_{m}\|_{\Hil}^{2}\leq\|a\|_{\Hil}^{2}$, so convergence of 
$\Bigl(\frac{5}{3}\Bigr)^{m}\DF_{m}^{a_{m}}(f)$  is equivalent to convergence of $\Bigl(\frac{5}{3}\Bigr)^{m}\DF_{m}(f)$ and thus to $f\in\domDF$.
\end{proof}

Of course one should expect that $\Bigl(\frac{5}{3}\Bigr)^{m}\DF_{m}^{a_{m}}(f)$  converges to $\DF^{a}(f)$, but we have only proved this under a condition akin to assuming $a\in\Hil$ is locally exact.  Note that in the classical (Euclidean) setting all $1$-forms are locally exact because the space is locally topologically trivial, but this is not the case on fractals.

\begin{defn}\label{defn:exact}
A $1$-form $a\in\Hil$  is called exact if there is $A\in\domDF$ such that $\partial A=a$.  It is locally exact if there is an open cover such that it is exact on the open sets.  Equivalently, it is locally exact if there is a finite partition of $\SG=\cup_{j} X_{w_{j}}$ of SG into cells $X_{w_{j}}=F_{w_{j}}(\SG)$ such that $a$ is exact on each cell, meaning there are $A_{w_{j}}\in\domDF$ so $a\mathds{1}_{w_{j}} =(\partial A_{w_{j}})\mathds{1}_{w_{j}}$ for all $j$.  We say $a$ is exact at scale $m$ if this is the smallest integer for which the  partition can be chosen to consist of $m$-cells.
\end{defn}

It is proved in~\cite{HR} that when $a$ is real-valued and exact there is a Coulomb gauge transformation which conjugates $\DF^{a}$ to $\DF$ and $\Mag^{a}$ to $\Delta$.  Specifically, one has from Corollary~5.6 of~\cite{HR}
\begin{gather}
	 \DF^{a}(f) = \DF(e^{iA}f) \label{eqn:gaugeforDF}\\
	\Mag^{a}f = e^{-iA} \Delta (e^{iA} f)\label{eqn:gaugeforMag}
	\end{gather}
In fact rather more can be obtained from the discussion at the end of Section~5 of~\cite{HR}, using the notion of a Coulomb gauge.

\begin{defn}
Suppose $a\in\Hil$ is real-valued. We say $a$ admits a Coulomb gauge if there is $e^{iA}\in\domDF$ such that $e^{-iA}\partial (e^{iA})=a$, and $a$ admits a local Coulomb gauge if this is true on the cells of a finite partition.
\end{defn}
\begin{rmk}
If $a$ admits a Coulomb gauge then it is locally exact, because $e^{iA}$ is uniformly continuous and thus has a logarithm in $\domDF$ on all sufficiently small cells.  However, having a Coulomb gauge is weaker than (global) exactness because it is possible to have $e^{iA}\in\domDF$ with $A$ locally but not globally in $\domDF$.  To see the distinction, suppose that $a$ is locally exact with $a=\partial A_{j}$ on cells $X_{w_{j}}$.  The $A_{w_{j}}$ are defined up to additive constants, and $a$ is exact if and only if we can choose these constants so $A=A_{w_{j}}$ on $X_{w_{j}}$  is continuous on SG.  By contrast, $a$ has a Coulomb gauge if we can choose the constants so that $e^{iA}=e^{iA_{w_{j}}}$ on $X_{w_{j}}$ is continuous on SG, so in this latter case we may permit jump discontinuities that are integer multiples of $2\pi$ at intersection points of the cells.
\end{rmk}

From Theorem~5.9 of~\cite{HR} both~\eqref{eqn:gaugeforDF} and~\eqref{eqn:gaugeforMag} are valid when $a$ admits a Coulomb gauge.  Note that this Corollary relies on the hypothesis that for connected open sets $U$,  $\partial f\mathds{1}_{U}=0$ implies $f$ is constant on $U$.  This is valid on SG because we can write $U$ as a connected union of cells, whence at any finite scale the cellular decomposition of $\|\cdot\|_{\Hil}$ allows us to assume the restriction of $\partial f$ to each cell is zero. Since each cell is self-similar to SG it suffices to note that if $\DF(f)=\|\partial f\|_{\Hil}^{2}=0$ then $f$ is constant by the properties of resistance forms.

It should be noted that when there is a Coulomb gauge we may immediately write a gauge transformation of $\DF_{m}^{a_{m}}$ and $\Mag_{m}^{a_{m}}$, because in this case the function $e^{iA}\in\domDF$ has an $m$-harmonic approximation (see~(A\ref{Aharmonic} for the definition). The $m$-harmonic approximation has the same values as $e^{iA}$ at points of $V_{m}$, so denoting it with $\bigl(e^{iA}\bigr)_{m}$ we can use to to write
\begin{equation*}
	e^{ia_{m}(e_{xy})} = \bigl( e^{iA(y)}\bigr)_{m} \bigl( e^{-iA(x)}\bigr)_{m} 
	\end{equation*}
for all $x\sim_{m}y$ in $V_{m}$, and therefore
\begin{equation}\label{eqn:exactgaugetransfforDFm}
	\DF_{m}^{a_{m}}(f)
	= \sum_{x,y: x\sim_{m} y} \Bigl| f(x)\bigl( e^{iA(x)}\bigr)_{m} - f(y)\bigl( e^{iA(y)}\bigr)_{m}  \Bigr|^{2}
	= \DF_{m}\bigl(e^{iA}f\bigr),
	\end{equation}
and similarly, for $x\in V_{m}$,
\begin{align}
	\Mag^{a_{m}}_{m} f(x)
	&= -  \bigl(e^{-iA(x)}\bigr)_{m} \sum_{y:y\sim_{m}x} \Bigl( f(x)\bigl(e^{iA(x)}\bigr)_{m}- f(y)\bigl(e^{iA(y)}\bigr)_{m} \Bigr)\\
	&=   e^{-iA}  \Delta_{m}\bigl(e^{iA}f\bigr).
	\end{align}

For forms that admit a local Coulomb gauge our graph magnetic energies converge to the magnetic energy on the fractal.
\begin{thm}\label{thm:DFamconvergestoDFa}
If $a\in\Hil$ is real-valued and has a local Coulomb gauge at scale $n$ then
\begin{equation*}
	\Bigl(\frac{5}{3}\Bigr)^{m}\DF_{m}^{a_{m}}(f)\to\DF^{a}(f) \text{ as } m\to\infty.
	\end{equation*}
\end{thm}
\begin{proof}
By hypothesis we may partition SG as $\cup_{|w|=n}F_{w}(\SG)$ and have  functions $e^{iA_{w}}\in\domDF$ such that
\begin{equation*}
	\DF^{a}(f)=\sum_{|w|=n} \DF_{X_{w}} \bigl(e^{iA_{w}} f|_{X_{w}} \bigr)
	\end{equation*}
where $ \DF_{X_{w}} $ is the Dirichlet form on the cell $X_{w}=F_{w}(\SG)$, so is just a rescaling of the global Dirichlet form.

On each cell the $m$-scale energy $\DF_{m}$ converges to $\DF$, so take $m>n$ sufficiently large that
\begin{equation*}
	 \DF_{X_{w},m} \bigl(e^{iA_{w}} f|_{X_{w}} \bigr) \leq  \DF_{X_{w}} \bigl(e^{iA_{w}} f|_{X_{w}} \bigr) \leq  \frac{\epsilon}{N}+ \DF_{X_{w}, m} \bigl(e^{iA_{w}} f|_{X_{w}} \bigr).
	\end{equation*}
where $N$ is the number of $n$-cells.
Now by~\eqref{eqn:exactgaugetransfforDFm} each of the $\DF_{X_{w},m} \bigl(e^{iA_{w}} f|_{X_{w}} \bigr)$ is that part of the sum for $\DF_{m}^{a_{m}}(f)$ which corresponds to the edges in $X_{w}$, so summing over the finite collection of cells in the truncated sum gives $\DF_{m}^{a_{m}}(f)$ and we have shown it is within $2\epsilon$ of $\DF^{a}(f)$.
\end{proof}
\begin{rmk}
We conjecture that Theorem~\ref{thm:DFamconvergestoDFa} holds without the restriction that $a$ admits a local Coulomb gauge.  Note, however, that we will also need the Coulomb gauge restriction to prove our results on the spectrum of $\Mag^{a}$ in Section~\ref{sec:spectra}, so little is lost by making this assumption here too.
\end{rmk}

\begin{thm}\label{thm:convergeofMagm}
Suppose $a\in\Hil$ is real-valued and has a local Coulomb gauge at scale $n$.  Then $f\in\dom(\Mag^{a})$ if and only if  $\frac{3}{2} 5^{m} \Mag_{m}^{a_{m}} f $ converges uniformly on $V_{\ast}\setminus V_{0}$ to a continuous function $\Phi$.  In this case the continuous extension of $\Phi$ to SG is $\Mag^{a}f$.
\end{thm}
\begin{proof}
First assume the uniform convergence to a continuous $\Phi$. For any $g\in\domDF$ that vanishes on $V_{0}$ define functions $h_{m}$ which are harmonic at scale $m$ and have values
\begin{equation*}
	h_{m}(x) =  \frac{3}{2} 5^{m} \bigl( \Mag_{m}^{a_{m}}f(x)\bigr) \overline{g(x)} \qquad\text{ for } x\in V_{m}\setminus V_0.
	\end{equation*}
Obviously $h_{m}(x)$ converges uniformly on SG to the continuous extension of $\Phi(x)\overline{g(x)}$.  What is more, the integral of the $m$-harmonic function which is $1$ at $x\in V_{m}\setminus V_{0}$ and zero at all other points of $V_{m}$ is $\frac{2}{3}3^{-m}$  so we may compute
\begin{equation*}
	\int h_{m}(x)\,d\mu = \Bigl(\frac{5}{3}\Bigr)^{m} \sum_{x\in V_{m}}\bigl( \Mag_{m}^{a_{m}}f(x)\bigr) \overline{g(x)} 
	\end{equation*}
Then~\eqref{eqn:DFmandMagfm} says that
\begin{equation*}
	\int h_{m}(x)\,d\mu = -\Bigl(\frac{5}{3}\Bigr)^{m} \DF_{m}^{a_{m}} (f,g)
	\end{equation*}
By Theorem~\ref{thm:DFamconvergestoDFa} and the parallelogram law the right side converges to $-\DF^{a}(f,g)$, and since the left side converges to $\int \Phi\bar{g}\,d\mu=\langle \Phi,g\rangle_{L^{2}(\mu)}$ and $g\in\domDF_{0}$ is arbitrary it must be that $f\in\dom(\Mag^{a})$ with $\Mag^{a}f$ being the continuous extension of $\Phi$ to SG.

Conversely we have $\DF^{a}(f,g)=-\langle \Mag^{a}f,g\rangle_{L^{2}(\mu)}$ for all $g\in\domDF_{0}$ and will make a careful choice of $g$.  Fix $x\in V_{\ast}\setminus V_{0}$ and $m\geq n$.  Since $a$ has a Coulomb gauge at scale $n$ we may find $e^{iA_{w}},e^{iA_{w'}}\in\domDF$ so that $a\mathds{1}_{w}=e^{-iA_{w}}\partial(e^{iA_{w}})\mathds{1}_{w}$ and $a\mathds{1}_{w'}=e^{-iA_{w'}}\partial(e^{i A_{w'}})\mathds{1}_{w'}$, where $F_{w}(\SG)$ and $F_{w'}(\SG)$ are the two $n$-cells that meet at $x$.  However the $e^{iA_{w}}$ and $e^{iA_{w'}}$ are defined only up to multiplicative constants constants of norm $1$, so we can arrange that they join continuously at $x$ and write both as $e^{iA}$.  Now let $\phi_{m}$ be the $m$-harmonic function which is equal $e^{iA(x)}$ at $x$ and zero at all other points of $V_{m}$ and define $\psi_{m}=e^{-iA}\phi_{m}$.  Note that both $\phi_{m}$ and $\psi_{m}$ are identically zero off $F_{w}(\SG)\cup F_{w'}(\SG)$, so the behavior of $A$ off this set does not affect $\psi_{m}$.  Since $\psi_{m}$ is a product of elements of $\domDF$ and is zero at $V_{0}$ it is in $\domDF_{0}$. 
Using this and the fact that $\psi_{m}$ is supported on the set where the gauge transformation is valid
\begin{equation*}
	-\langle \Mag^{a}f,\psi_{m}\rangle_{L^{2}(\mu)}
	=\DF^{a}(f,\psi_{m})
	=\DF(e^{iA}f,e^{iA}\psi_{m})
	= \DF(e^{iA}f,\phi_{m})
	\end{equation*}
but $\phi_{m}$ is $m$-harmonic, so 
\begin{equation*}
	\DF(e^{iA}f,\phi_{m})
	=\bigl(\frac{5}{3}\bigr)^{m}\DF_{m}(e^{iA}f,\phi_{m})
	= \bigl(\frac{5}{3}\bigr)^{m}\DF_{m}^{a_{m}}(f,\psi_{m})
	\end{equation*}
and inserting~\eqref{eqn:DFmandMagfm} we have only the terms involving $x$, so
\begin{equation*}
	3^{m} \langle \Mag^{a}f,\psi_{m}\rangle_{L^{2}(\mu)} = 5^{m} \Mag_{m}^{a_{m}}f(x).
	\end{equation*}
We assumed $\Mag^{a}f$ was continuous, and it is obvious the support of the $\psi_{m}$ converges to $x$, so the proof will be complete if we show $3^{m}\int \psi_{m}\,d\mu\to\frac{2}{3}$.  However $e^{iA}$ is continuous, so its restriction to the support of $\psi_{m}$ converges uniformly to $e^{iA(x)}$ as $m\to\infty$.  If $\chi_{m}$ denotes the  the harmonic function which is $1$ at $x$ and zero on the other points of $V_{m}$ then we conclude $\psi_{m}-\chi_{m}$  converges uniformly to zero.  Moreover $3^{m}\int \chi_{m}\,d\mu=\frac{2}{3}$ for all $m$ by elementary symmetry considerations, so the proof is complete.
\end{proof}

\begin{thm}
Suppose  $a\in\Hil$ is real-valued and admits a local Coulomb gauge at scale $n$. If $f\in\dom(\Mag^{a})$, then the magnetic normal derivative
\begin{equation*}
	d^{a}f(p) = \lim_{m\to\infty} \Bigl( \frac{5}{3}\bigr)^{m} \sum_{x\sim_{m}p} \bigl( f(p)- e^{i(A_{p}(x)-A_{p}(p))} f(x) \bigr)
	\end{equation*}
exists at each $p\in V_{0}$ and for $g\in\domDF$ we have the Gauss-Green formula
\begin{equation}\label{eqn:magGG}
	\DF^{a}(f,g) = -\langle \Mag^{a} f,g \rangle_{L^{2}(\mu)} + \sum_{x\in V_{0}} \bigl(d^{a}f(p) \bigr)\overline {g(p)}.
	\end{equation}
If, in addition, there is $A_{p}\in\domDF$ such that $\partial A_{p}=a$ on a neighborhood of $p$, and the usual normal derivative $d A_{p}(p)$ exists, then $df(p)$ exists and
\begin{equation}\label{eqn:magnormalderivfromusual}
	d^{a}f(p)= e^{-iA_{p}(p)}d \bigl( fe^{iA_{p}} \bigr) = df(p) + if(p)dA_{p}(p)
	\end{equation}
\end{thm}
\begin{proof}
Fix $g\in\domDF$.  For each $m$ and each $p\in V_{0}$ use the construction of $\psi_{m}$ from the proof of  Theorem~\ref{thm:convergeofMagm} to obtain a function $\psi_{m}^{p}$ which is $1$ at $p$, zero at all other points of $V_{m}$ and such that if $e^{iA_{p}}$ is the local Coulomb gauge at $p$ then $e^{iA_{p}}\psi_{m}^{p}$ is $m$-harmonic.  Let $g_{m}=\sum_{p\in V_{0}} g(p)\psi_{m}^{p}$.  Then $g-g_{m}\in\domDF_{0}$ and therefore $\DF^{a}(f,g-g_{m})=-\langle \Mag^{a}f,g-g_{m}\rangle_{L^{2}(\mu)}$.  Since $\Mag^{a}f$ is continuous and $g-g_{m}\to g$ in $L^{2}(\mu)$ we find that $\DF^{a}(f,g_{m})$ converges.  For large enough $m$ the gauge transform and the definition of $\psi_{m}^{p}$ imply
\begin{align*}
	\DF^{a}(f,g_{m})
	&=\sum_{p}\overline{g(p)}\DF\bigl(e^{iA_{p}}f,e^{iA_{p}}\psi_{m}^{p}\bigr) \\
	&=\Bigl( \frac{5}{3}\bigr)^{m}\sum_{p}\overline{g(p)}\DF_{m} \bigl( e^{iA_{p}}f,e^{iA_{p}}\psi_{m}^{p}\bigr)\\
	&=\Bigl( \frac{5}{3}\bigr)^{m} \sum_{p}\overline{g(p)}\sum_{x\sim_{m}p} \bigl( f(p)- e^{i(A_{p}(x)-A_{p}(p))} f(x) \bigr)
	\end{align*}
so that the magnetic normal derivative exists and~\eqref{eqn:magGG} holds.

When $dA_{p}$ exists we have $A_{p}(x)-A_{p}(p) = -\Bigl(\frac{3}{5}\Bigr)^{m}dA_{p}(p) + o \Bigl(\frac{3}{5}\Bigr)^{m}$ for both $x\sim_{m}p$.  Thus we compute
\begin{align*}
	df(p)
	&= \lim_{m\to\infty} \Bigl( \frac{5}{3}\bigr)^{m} \sum_{x\sim_{m}p} \bigl( f(p) -  f(x)\bigr) \\
	&=\lim_{m\to\infty} \Bigl( \frac{5}{3}\bigr)^{m} \sum_{x\sim_{m}p} \Bigl(\bigl( f(p) -  f(x)e^{i(A_{p}(x)-A_{p}(p))}\bigr) + f(x)\bigl( e^{i(A_{p}(x)-A_{p}(p))}-1\bigr)\Bigr) \\
	&= d^{a}f(p) - if(p) dA_{p}(p)
	\end{align*}
which gives the second conclusion of the theorem.
\end{proof}

It is apparent that we can localize the magnetic Gauss-Green formula to any cell.  Doing so allows us to give necessary and sufficient conditions for defining a function in $\dom(\Mag^{a})$ piecewise.

\begin{thm}\label{thm:gluing}
Suppose $a\in\Hil$ is real-valued and admits a local Coulomb gauge.  Let $X_{1}=F_{w_{1}}(\SG)$ and $X_{2}=F_{w_{2}}(SG)$ be two cells with $X_{1}\cap X_{2}=\{x\}$ and assume we have functions $f_{j}$ and $u_{j}$ from $\domDF|_{X_{j}}$ such that $\Mag^{a}f_{j}=u_{j}$, $j=1,2$.  In order that the piecewise functions  $f=f_{j}$ on $X_{j}$ and $u=u_{j}$ on $X_{j}$ for $j=1,2$ satisfy $\Mag^{a}f=u$ it is necessary and sufficient that both are  continuous, $f_{1}(x)=f_{2}(x)$ and $u_{1}(x)=u_{2}(x)$, and also that $d^{a} f_{1}(x)+d^{a}f_{2}(x)=0$.
\end{thm}
\begin{proof}
The role of the continuity assumption is elementary, so we focus on the condition on $d^{a}$.
By localizing~\eqref{eqn:magGG} to $X_{1}$ and $X_{2}$ we may write the hypothesis $\Mag^{a}u_{j}=f_{j}$ as
\begin{equation}\label{eqn:gluingstep}
	\DF^{a}_{X_{j}} (f_{j},g) = -\langle u_{j},g \rangle_{L^{2}(\mu,X_{j})} + \sum_{p\in V_{0}} \bigl(d^{a}f_{j}(F_{w_{j}}(p)) \bigr) \overline{g(F_{w_{j}}(p))}
	\end{equation}
for $j=1,2$.
Similarly, $\Mag^{a}u=f$ on the union means that for functions $g$ which vanish on $\bigl(F_{w_{1}}(V_{0})\cup F_{w_{2}}(V_{0})\bigr)\setminus\{x\}$ we have
\begin{equation*}
	\DF^{a}_{X_{1}\cup X_{2}} (f,g)
	= -\langle u,g \rangle_{L^{2}(\mu,X_{1}\cup X_{2})}.
	\end{equation*}
Comparing this to the sum of~\eqref{eqn:gluingstep} for $j=1,2$ we see that they are the same if and only if all the terms from the sums over $V_0$ vanish.  Our hypothesis on $g$ ensures these sums only contain the two terms at $x$, so the quantity which must vanish is $(d^{a}f_{1}(x)+d^{a}f_{2}(x))\overline{g(x)}$, and $g(x)$ can take any value.
\end{proof}


We conclude this section with a discussion of the structure of the subspace of exact forms on SG and its complementary subspace in $\Hil$.  Recall that the exact forms are the image of the map $\partial:\domDF\to\Hil$.  Since $\|\partial f\|_{\Hil}^{2}=\DF(f)$ and $\domDF$ modulo constants is a Hilbert space, the exact $1$-forms are a complete, hence closed, subspace of $\Hil$.  We write $P$ for the projection onto the exact forms and $P^{\perp}$ for the orthogonal projection. It is proven in~\cite{IRT} that $P\Hil_{m}$ is the space obtained by applying $\partial$ to the $m$-harmonic functions, while $P^{\perp}\Hil_{m}$ is the space of $m$-harmonic $1$-forms.  A $1$-form is $m$-harmonic if on each $m$-cell $X_{w}=F_{w}(SG)$ it is $(\partial h_{w})\mathds{1}{w}$ for some $m$-harmonic function $h_{w}$, and for any point $x\in V_{m}$ the sum of the normal derivatives $\sum_{w} dh_{w}(x)$ over the cells meeting at $x$ is zero.

The self-similarity of the space $\Hil_{m}$ ensures we may understand the structure of $\Hil_{m}$ by studying the structure of $\Hil_{1}$.  This is generated by the harmonic functions modulo constants on the $1$-cells.  It is convenient to incorporate the condition on constants by assuming the harmonic functions have mean zero, so the sum of the values at points $F_{j}(V_{0})$ is zero for each $j\in\{0,1,2\}$.  One can then check that $P\Hil_{1}$ is $5$-dimensional.  In fact, the $5$-dimensional space  generated by $1$-harmonic functions that are mean-zero on SG  can be made mean-zero on each $F_{j}(\SG)$ by subtracting an appropriate mean-zero function that is harmonic on all of SG, so this space decomposes into the $2$-dimensional space $\Hil_{0}=P\Hil_{0}$ and a $3$-dimensional complement. The remaining space, $P^{\perp}\Hil_{1}$ is $1$-dimensional and corresponds to a loop around the central hole. We let $b\in\Hil_{1}$ be the element with counterclockwise orientation shown in  Figure~\ref{fig:loopelement}(a), multiplied by $1/\sqrt{30}$ so that $\|b\|_{\Hil}=1$.  It is also convenient to choose harmonic functions  on the $1$-cells as shown in Figure~\ref{fig:loopelement}(b) such that applying $\partial$ gives $\sqrt{30}b$.  Although this latter is not a function on SG it is a function $B$ on the disjoint union $\sqcup_{j=0,1,2}F_{j}(SG)$.

\begin{figure}[htb]
\begin{picture}(105.6,90)(0,-3)
\setlength{\unitlength}{.23pt} \Spic{2}{32}{0}{0}
\put(72,-40){$-1$}  \put(252,-40){$-1$}
\put(-13,70){$-1$} \put(152,70){$2$} \put(212,70){$2$} \put(345,70){$-1$}
\put(182,126){$2$}
 \put(248,228){$-1$}
\put(82,228){$-1$}
\end{picture}
\ \ \ \ \ \ \ \ 
\begin{picture}(105.6,90)(0,-3)
\setlength{\unitlength}{.20pt}
\Separatedfirstlevel{32}
\put(-18,-40){$0$} \put(442,-40){$0$} \put(212,380){$0$}
\put(185,-40){$1$} \put(220 ,-40){$-1$} 
\put(30,155){$-1$} \put(375,155){$1$}
\put(100,210){$1$} \put(320,210){$-1$}
\end{picture}
\caption{(a) The $1$-form $\sqrt{30}b$, with orientation clockwise around each $1$-cell, hence counterclockwise around the central hole,  and
	 (b) The harmonic function $B$ on disjoint $1$-cells.}\label{fig:loopelement}
\end{figure}
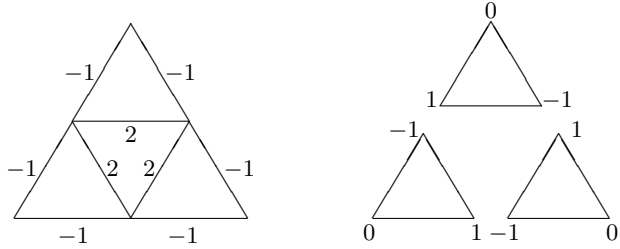

It is apparent that the set $\{b\circ F_{w}\}$ of $1$-forms span the space of harmonic forms $P^{\perp}\Hil$.  If $b\circ F_{w}$ and $b\circ F_{w'}$  are from disjoint cells then the direct sum decomposition in~(F\ref{Fcelldecomp}) implies they are orthogonal,  and by computing $\Tr_{0}b=0$ from the formula in Definition~\ref{defn:trace} we find $b\circ F_{w}$ and $b\circ F_{w'}$  are orthogonal if $|w|\neq|w'|$. Thus $\{b\circ F_{w}\}$ is an orthogonal basis for $P^{\perp}\Hil$, and for real values $\beta_{w}$
\begin{equation}\label{eqn:normofharmonicform}
	\Bigl\| \sum_{m=1}^{\infty}\sum_{|w|=m} \beta_{w}b\circ F_{w} \Bigr\|_{\Hil}^{2} = \sum_{m=1}^{\infty}\Bigl( \frac{5}{3}\Bigr)^{m}\sum_{|w|=m} \beta_{w}^{2}
	\end{equation} 
if the latter series converges.  Moreover, $a\in\Hil$ is locally exact if and only if $P^{\perp}a$ is a series of this type with only finitely many terms, it has Coulomb gauge if and only if all $\beta_{w}$ in the series for $P^{\perp}a$ are integer multiples of $2\pi$, and it has scale $n$ Coulomb gauge if and only if all $\beta_{w}$ in the series for $P^{\perp}a$ which have $|w|>n$ are integer multiples of $2\pi$.  Note that the final point and~\eqref{eqn:normofharmonicform} gives another proof that every form admiting a local Coulomb gauge is locally exact, though not necessarily at the same scale.


\section{Spectra of magnetic operators with local Coulomb gauge}\label{sec:spectra}

In this section we study the spectrum of Dirichlet magnetic operators $\Mag^{a}$, which we know from Theorem~\ref{thm:mainresultofHR}  is pure point.   Our approach relies heavily on the spectral decimation property of the Laplacian on SG~\cite{RammalToulouse,Shima,FukushimaShima} and associated properties of the eigenfunctions~\cite{DSV,Kig1998}.  Spectral decimation says that if $f$ is an eigenfunction of $\Delta$ on SG then there is $m_{0}$ (called the generation of birth) and a sequence $\{\lambda_{m}\}_{m_{0}}^{\infty}$ such that $\Delta_{m}f=\lambda_{m}f$ for all $m\geq m_{0}$.  The sequence $\{\lambda_{m}\}$ is related to the eigenvalue by $\lambda_{m}(5-\lambda_{m})=\lambda_{m-1}$ and $\frac{3}{2}\lim 5^{m}\lambda_{m}=\lambda$.  One  way to view this graph eigenfunction equation is as follows:  if on each $m$-cell $F_{w}(\SG)$ we have $f_{w}$ such that $\Delta f_{w}=\lambda f_{w}$ then defining $f$ piecewise to be $f_{w}$ on  $F_{w}(\SG)$  we have $\Delta f=\lambda f$ if and only if $f$ is continuous and $\Delta_{m}f=\lambda_{m}f$.  Comparing this to the usual gluing property we see that the discrete eigenfunction equation encodes that the normal derivatives sum to zero at the points of $V_{m}$.  The equivalence of these conditions may also be verified using the explicit formulas for the normal derivatives from~\cite{DRS}.

We wish to study the spectrum of $\Mag^{a}$ via the finite approximations $\Mag^{a_{m}}_{m}$, so in light of the results of the previous section it makes sense to only consider real-valued $a\in\Hil$  which admit a local Coulomb gauge at scale $n$.   By the discussion following Definition~\ref{defn:exact} we may also assume that $Pa=0$, because we can gauge transform to remove this part of $a$.  Doing so will not change the eigenvalues of $\Mag^{a}$ and will simply conjugate the eigenfunctions.   Under these assumptions let $m\geq n$ and $e^{iA_{w}}$ be the gauge transform on the $m$-cell $F_{w}(\SG)$.  Then $u_{w}$ satisfies $\Mag^{a}u_{w}=\lambda u_{w}$ on $F_{w}(\SG)$ if and only if $f_{w}=e^{iA_{w}}u_{w}$ and $\Delta f_{w}=\lambda f_{w}$ on the cell.  The condition for gluing the $u_{w}$ into a piecewise defined eigenfunction with $\Mag^{a}u=\lambda u$  is that they join continuously and $\sum_{w}d^{a}u_{w}(p)=0$,  where the sum is over the cells meeting at $p\in V_{m}$.  From $Pa=0$ we have $\sum_{w}dA_{w}(p)=0$ for all $p$, so by~\eqref{eqn:magnormalderivfromusual} our condition becomes $\sum_{w}e^{-iA_{w}(p)}df_{w}(p)=\sum_{w}du_{w}=\sum_{w}d^{a}u_{w}=0$.  But $\Delta e^{-iA_{w}(p)}f_{w}(x)=\lambda e^{-iA_{w}(p)}f_{w}(x)$,  so the normal derivatives sum to zero at $p$ if and only if $\Delta_{m}e^{-iA_{w}(p)}f_{w}(p)=\lambda_{m}e^{-iA_{w}(p)}f_{w}(p)$, which is precisely $\Mag_{m}^{a_{m}}u_{w}=\lambda_{m} u_{w}$.  Thus we can study $\Mag^{a}u=\lambda u$ by examining $\Mag_{m}^{a_{m}}u=\lambda_{m}u$ for $m\geq n$.

As described at the end of the previous section, the assumptions we have on $a$ imply that there are real numbers $\beta_{w}$ with $\beta_{w}\in2\pi\mathbb{Z}$ for $|w|>n$, such that 
\begin{gather}
	a=\sum_{m=1}^{\infty}\sum_{|w|=m} \beta_{w}b\circ F_{w}, \label{eqn:nonexactmag}\\
	\|a\|_{\Hil}^{2}=\sum_{m=1}^{\infty}\Bigl( \frac{5}{3}\Bigr)^{m}\sum_{|w|=m} \beta_{w}^{2}<\infty. \notag
	\end{gather}
Since all terms in this expression are self-similar it is clear that a significant step is to understand the spectrum of $\Mag^{\beta b}$, in which case we can look at $\Mag_{1}^{\beta b}$.

\begin{figure}
[t]
\includegraphics[width=.65\textwidth]{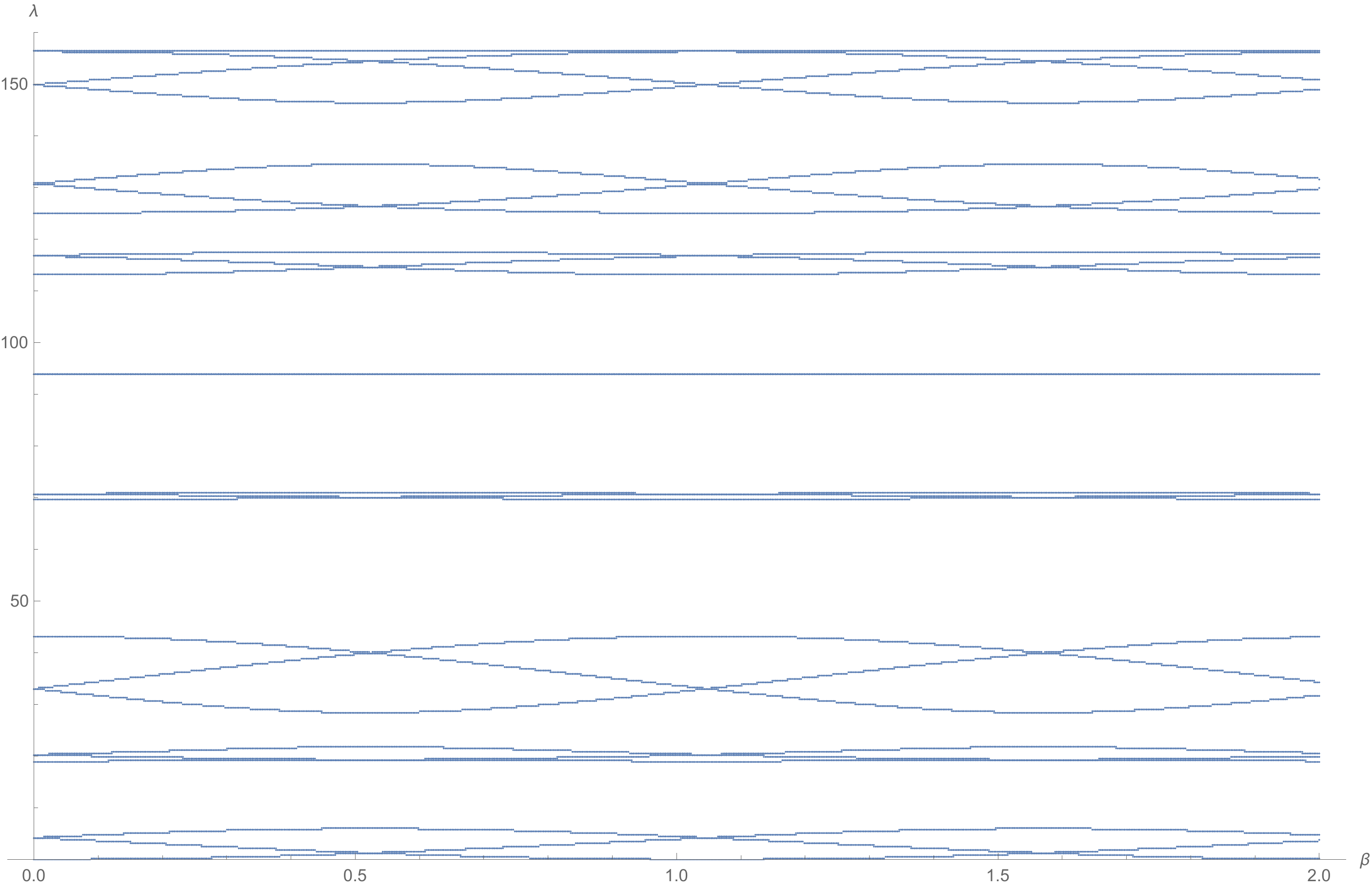}
\\[12pt]
\includegraphics[width=.65\textwidth]{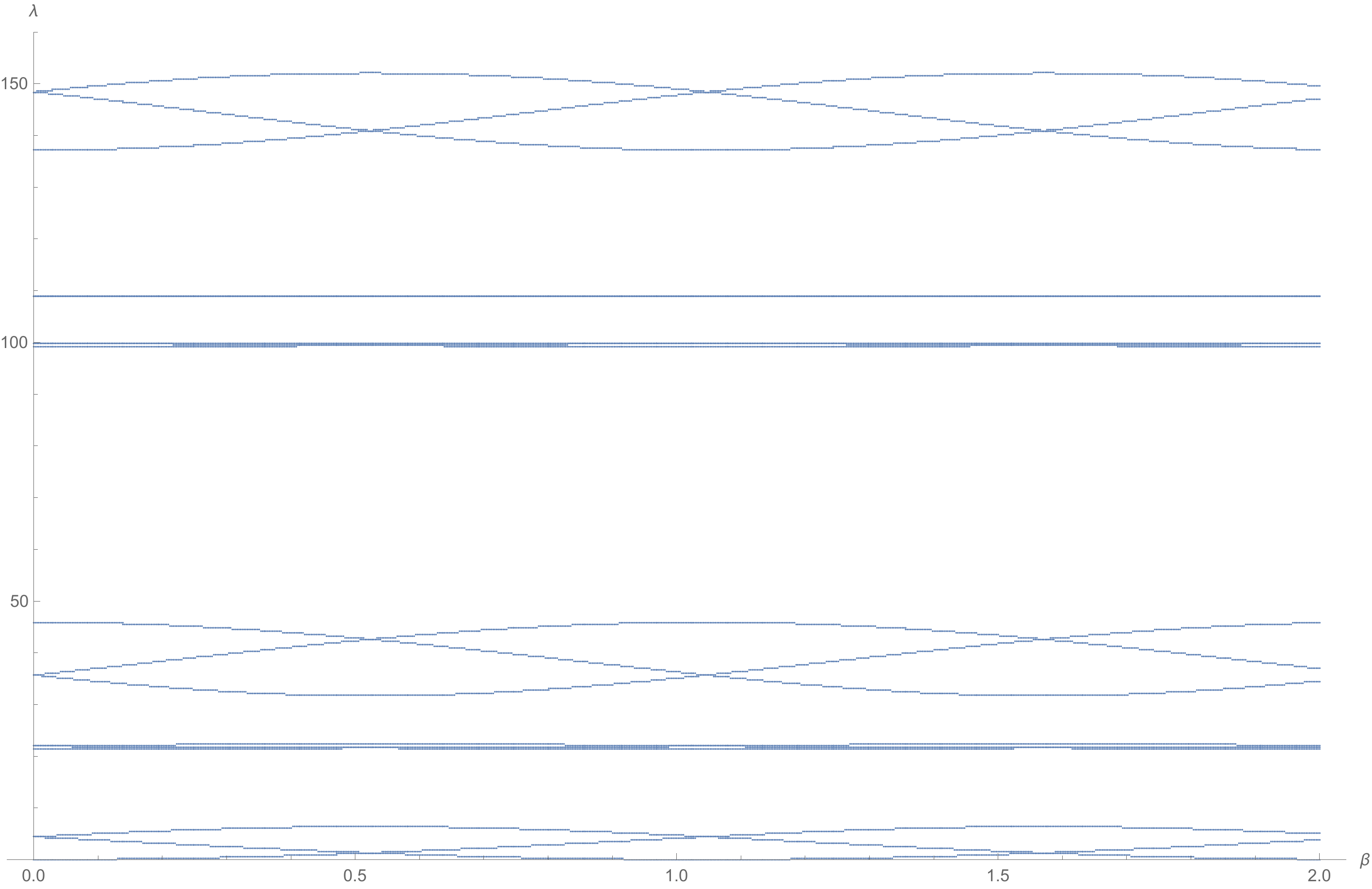}
\\[12pt]
\includegraphics[width=.65\textwidth]{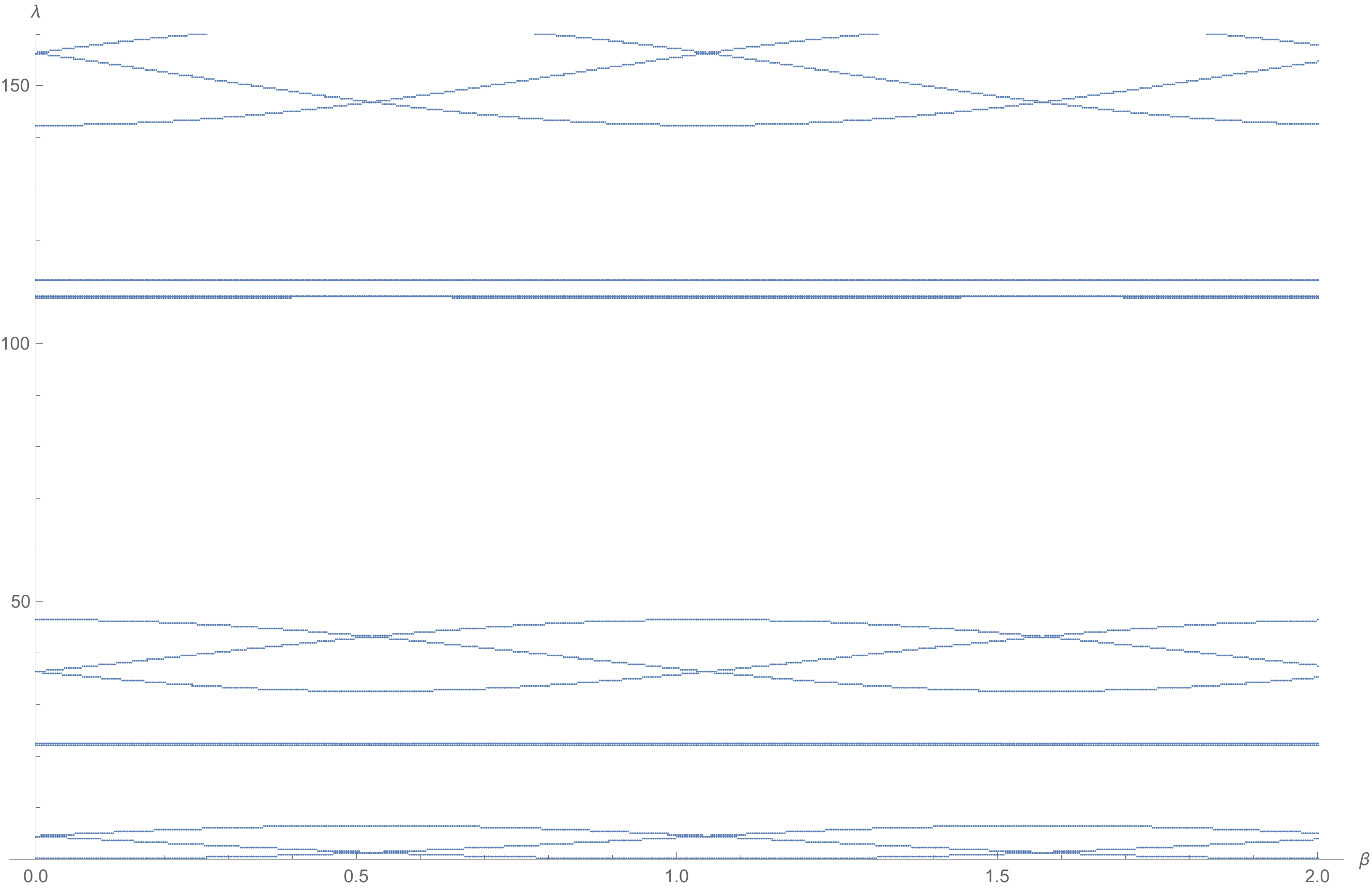}
\caption{Eigenvalues less than $160$ and $0\leq \beta\leq 2$ for the (from top to bottom) 4th, 5th, and 6th level approximation to $\mathcal M^{\beta b}$ } \label{SpectrumGraphs}
\end{figure}


The results of some numerical investigations into the spectrum of $\Mag^{b}$ are shown in Figure~\ref{SpectrumGraphs}. One can see the structure of the spectrum inherited from the spectral decimation process, which copies and expands the spectrum with each level of approximation.

Of particular note is the existence of many eigenvalues that do not vary with $\beta$, and are therefore independent of the field. These can be seen in Figure \ref{SpectrumGraphs} as horizontal lines.  This pattern persists for more complicated magnetic operators $\Mag^{a}$ with local Coulomb gauge: when $m$ is sufficiently large we find that  $\Mag_{m}^{a_{m}}$ has a large number of eigenvalues that are the same as those of $\Delta_{m}$.  This turns out to be a straightforward consequence of the structure of the eigenfunctions of the Laplacian.

\begin{thm}\label{thm:conjugateLapefns}
 Suppose $f$ is an eigenfunction of $\Delta$ with eigenvalue $\lambda$ and the support of  $f$ is a finite union of cells $\cup X_{k}$ on which $a$ has a Coulomb gauge, so there is $e^{iA}\in\domDF$ such that  $e^{-iA}(\partial e^{iA})\mathds{1}_{\cup X_{k}}=a\mathds{1}_{\cup X_{k}}$.  Then $fe^{-iA}$ is an eigenfunction of $\Mag^{a}$ with eigenvalue $\lambda$.
\end{thm}
\begin{proof}
This is a direct computation from the validity of the gauge transformation on $\cup X_{k}$, because for $g\in\domDF_{0}$
\begin{equation*}
	\DF^{a}(fe^{-iA},g)=\DF(f,e^{iA}g)
	=-\lambda\langle f, e^{iA}g\rangle
	=-\lambda\langle fe^{-iA},g\rangle. \qedhere
	\end{equation*}
\end{proof}
\begin{rmk}
This result can also be thought of in terms of the gluing result in Theorem~\ref{thm:gluing}.  By construction $fe^{-iA}$ satisfies the eigenfunction equation for $\Mag^{a}$ on $\cup X_{k}$.  From the fact that $f$ is identically zero outside $\cup X_{k}$ we see that $df$ must be zero at the boundary points of $\cup X_{k}$.  Using~\eqref{eqn:magnormalderivfromusual} with $f=df=0$ on the boundary of $\cup X_{k}$ we have $fe^{-iA}=d^{a}(fe^{-iA})=0$ there also, so extending $fe^{-iA}$ by zero gives a smooth solution of the eigenfunction equation on SG.
\end{rmk}

In order to see why this result determines many eigenfunctions of $\Mag^{a}$ we need some more consequences of the spectral decimation method, particularly those from~\cite{DSV,Kig1998}. Our presentation of them follows the elementary exposition in~\cite{Strichartzbook}, except that our bases for the $5$-series eigenspaces are more like those in~\cite{DSV}.  In order to describe these bases we define a chain of $m$-cells to be a sequence $X_{k}=F_{w_{k}}(\SG)$, $k=1,\dotsc, K$ such that $|w_{k}|=m$  for all $k$ and $X_{k}\cap X_{k+1}=\{x_{k}\}$ is a sequence of $K-1$ distinct points from $V_{m}$. We say the chain is simple if $X_{k}\cap X_{k'}=\emptyset$ unless $|k-k'|\leq1$.

\begin{enumerate}[({S}1)]
\item\label{S:specdec} For a Dirichlet eigenvalue $\lambda$ of $\Delta$ on SG with eigenfunction $f$, let $m(\lambda)$ be its generation of birth and $\lambda_{m}$ be the spectral  decimation sequence, so  $\Delta_{m}f=\lambda_{m}f$, $\lambda_{m}(5-\lambda_{m})=\lambda_{m-1}$ and $\frac{3}{2}\lim 5^{m}\lambda_{m}=\lambda$.  Then $\lambda_{m(\lambda_{0})}\in\{2,5,6\}$ and $\lambda_{m}\not\in\{2,5,6\}$ for $m>m(\lambda)$.  We let $\sigma_{s}=\{\lambda: \lambda_{m(\lambda_{0})}=s\}$ for $s=2,5,6$, and call these the $2$, $5$, and $6$ series eigenfunctions.
\item\label{S:fixation} From the preceding, $\lambda_{m}=\frac{1}{2}\bigl(5\pm\sqrt{25-4\lambda_{m-1}} \bigr)=\Phi_{\pm}(\lambda_{m-1})$.  For convergence of $5^{m}\lambda_{m}$ the positive root can occur at most finitely often, so there is $m_{1}(\lambda)$ called the generation of fixation such that $\lambda_{m}=\Phi_{-}(\lambda_{m-1})$ for all $m>m_{1}$.  Writing $\Phi_{-}^{\circ m}$ for the $m$-fold composition, the function $\mathcal{R}(\tau)=\lim_{m} 5^{m}\Phi_{-}^{\circ m}(\tau)$ is analytic, $\mathcal{R}(0)=0$ and $\mathcal{R}'(0)\neq0$.  Knowing the generation of fixation the eigenvalue is $\lambda=5^{m_{1}}\mathcal{R}(\lambda_{m_{1}})$.
\item  If $\lambda\in\sigma_{2}$ then $m(\lambda)=1$, its eigenspace is $1$-dimensional, and the eigenfunctions are fully symmetric under the dihedral symmetry group of the triangle.  
\item If $\lambda\in\sigma_{5}$ then $m(\lambda)\geq1$.  All eigenfunctions vanish on $V_{m(\lambda)-1}$ and the eigenspace has dimension $\frac{1}{2}\bigl(3^{m(\lambda)-1}+3\bigr)$.  There is a basis for the $5$-series eigenfunctions in which each is supported on a simple chain of $(m(\lambda)-1)$-cells in which $X_{k_{1}}$ and $X_{k_{K}}$ contain distinct points of $V_{0}$.
\item If $\lambda\in\sigma_{6}$ then  $m(\lambda)\geq2$.  The eigenspace has dimension $\frac{1}{2}\bigl(3^{m(\lambda)}-3\bigr)$, and there is a basis in which each eigenfunction is supported on the union of two $(m(\lambda)-1)$-cells meeting at a point of $V_{m(\lambda)-1}\setminus V_{0}$.
\end{enumerate}

A small comment about the $5$-series basis is in order.  With generation of birth $m+1$ there is a function supported on an $m$-cell with the following property: given an $m$-chain with both ends on $V_{0}$ there is an arrangement of copies of the function along the cells in the chain such that the  
resulting function extends smoothly  by $0$ to give a $5$-series eigenfunction on SG.  This arrangement is unique up to multiplying the eigenfunction by a scalar.  In~\cite{DSV} a basis is given in which each eigenfunction is supported on an $m$-cell chain from $p_{0}$ to either $p_{1}$ or $p_{2}$, but the chains given are not simple. In particular it follows from~\cite{DSV} that the number of  $m$-cell chains between two points of $V_{0}$ is $\frac{1}{2}\bigl(3^{m-1}+1\bigr)$.  Observe that each simple $m$-chain determines an $(m-1)$-chain by taking the parent cells of the $m$-cells in the chain. Conversely an $(m-1)$-chain determines a simple $m$-chain by taking, in each $(m-1)$-cell $X_{k}$, the two $m$-cells which form the shortest $m$-cell chain from $x_{k-1}$ to $x_{k}$. From this bijection between simple $m$-chains and $(m-1)$-chains we see that the number of simple $m$-cell chains between two specified points of $V_{0}$ is $\frac{1}{2}\bigl(3^{m-2}+1\bigr)$, and therefore the number of such chains joining pairs of points from $V_{0}$ is $\frac{1}{2}\bigl(3^{m-1}+3\bigr)$, which is the dimension of a $5$-series eigenspace with generation of birth $m$.  Moreover it is easy to prove inductively that the eigenfunctions corresponding to these chains are linearly independent.  When $m=2$ this can be done by hand (as was done in~\cite{DSV}).  For the inductive step observe that if a linear combination of eigenfunctions corresponding to simple $m$-chains is zero then it is zero on each cell $F_{j}(X)$, $j=0,1,2$.  Then precomposing the piece on $F_{j}(X)$ with $F_{j}^{-1}$ gives a vanishing linear combination of eigenfunctions corresponding to $(m-1)$-chains, and these are linearly independent by the inductive hypothesis.

\begin{thm}\label{thm:spectrumofMaga}
If $a$ is a real-valued form with local Coulomb gauge at scale $n$  and $\lambda$ is a Laplacian eigenvalue with generation of birth $m(\lambda)>n$ then $\lambda$ is also an eigenvalue of $\Mag^{a}$, and the corresponding eigenfunction is obtained from the Laplacian eigenfunction by a gauge transformation. 
\end{thm}
\begin{proof}
If $a$ is as described then on every $n$-cell $F_{w}(\SG)$ we have a gauge function $e^{iA_{w}}$, which is determined up to a multiplicative constant.  For $\lambda$ as described the Laplacian eigenfunction is supported either on simple chain of $(m(\lambda)-1)\geq n$ cells, or on the union of two $(m(\lambda)-1)$ cells, which we denote $X_{k}=F_{w_{k}}(\SG)$.  In either case simplicity of the chain ensures we may choose the values $e^{iA_{w}(x_{k})}$, where $x_{k}=X_{k}\cap X_{k+1}$, so that $e^{iA}=e^{iA_{w_{k}}}$ on $X_{k}$ is continuous, hence a Coulomb gauge on $\cup X_{k}$.  The result then follows from Theorem~\ref{thm:conjugateLapefns}.
\end{proof}

\begin{cor}\label{cor:spectasymptotics}
If $a$ is a real-valued form with local Coulomb gauge at scale $n$ then $\Mag^{a}$ has the same spectral asymptotics as $\Delta$.  Specifically, let $\rho^{a}(x)$ be the counting function of $\Mag^{a}$, so $\rho^{a}(x)=\#\{\lambda\in\sigma_{D}:\lambda\leq x\}$. There is a non-constant periodic function $\chi$ of period $\log 5$ such that 
\begin{equation*}
	\lim_{x\to\infty} \rho^{a}(x)x^{-\log3/\log5} - \chi(\log x) = 0.
	\end{equation*}
The function $\chi$ is independent of $a$, so is the same as that occuring for the Laplacian spectrum.
\end{cor}
\begin{proof}
For $a=0$ this is simply the spectral asymptotic for the Laplacian, and follows from a more general analysis in~\cite{KigLap}.  When $a\neq0$ the result follows from the fact that eigenvalues with generation of birth less than $n$ make an asymptotically small contribution to the spectrum.   To make this precise we reason as follows.  

The eigenvalues and eigenfunctions of $\Mag^{m}$ obey spectral decimation for all sufficiently large $m$, so for each eigenvalue $\lambda$ there is a sequence $\lambda_{m}$, $m\geq m_{0}$ as in~(S\ref{S:specdec})  and the eigenvalue is determined at the generation of fixation as described in~(S\ref{S:fixation}).  Following this line of reasoning, for a specified $x$ there is $m_{1}$ comparable to $\log x$ such that all eigenvalues $\lambda\leq x$ are of the form $5^{m_{1}}\Phi(\lambda_{m_{1}})$ for $\lambda_{m_{1}}$ an eigenvalue of $\Mag_{m_{1}}^{a_{m_{1}}}$. Hence it suffices to know what proportion of the eigenvalues of $\Mag_{m_{1}}^{a_{m_{1}}}$ have generation of birth $\leq n$.  At each $m$ the number of newly born eigenvalues is comparable to $3^{m}$, and these split according to the positive and negative roots in the spectral decimation to give a multiple of $2^{m_{1}-m}3^{m}$ eigenvalues at the generation of fixation, so the number of eigenvalues born before $n$ but fixed at $m_{1}$ is comparable to $3^{n}$, while the total number fixed at $m_{1}$ is comparable to $3^{m_{1}}$.  Thus the proportion of eigenvalues of $\Mag^{a}$ that differ from those of $\Delta$ and are less than $x$ is bounded by a multiple of $3^{n}/x$ for large $x$, and goes to zero as $x\to\infty$.
\end{proof}

Theorem~\ref{thm:spectrumofMaga} also gives all of the spectrum of $e^{i\beta b}$ except that born at generation $1$.  We can get the rest by direct computation.  If we label the points of $V_{1}\setminus V_{0}$ as $q_{j}$, $j=0,1,2$, then symmetry suggests we ought to have eigenfunctions $f_{k}$ of $\Mag_{1}^{a_{1}}$ with values on $V_{1}$ given by $f_{k}(q_{j})=e^{ijk2\pi/3}$.  Indeed 
\begin{equation*}
	\Mag_{1}^{\beta b}f_{k}=\Bigl( 4 - 2\cos \Bigl( \frac{2\pi k}{3} + \frac{2\beta}{\sqrt{30}} \Bigr) \Bigr)\, f_{k}
	\end{equation*}
from which we can determine the eigenfunctions by applying $5\Phi$.  Ideally we would like to be able to use this information to compute the bottom of the spectrum for $\Mag^{a}$ in the case where $a$ is given by~\eqref{eqn:nonexactmag}, at least in some special cases, but unfortunately we do not know how to do this.


\section{Spectrum of $\Mag^{\beta b}$ via the ladder fractafold}\label{sec:ladderfractafold}

An alternative approach to the problem of determining the spectrum of $\Mag^{\beta b}$ is to lift the problem to a periodic version on a suitable covering space using a technique from~\cite{ST12}.  To avoid merely repeating the results of the previous section we illustrate this method by computing the spectrum of the Neumann magnetic operator.

The space we use is called the Ladder Fractafold based on the Sierpinski gasket, and is denoted LF.  \cite{ST12} gives a general method for analyzing the spectrum of a fractafold constructed by gluing copies of SG arranged according to a graph.    For LF,  let the vertices of a graph $\Gamma_0$  be three copies of $\mathbb{Z}$, labelled  $\{x_{k+\frac{1}{2}}\}$, $\{w_k\}$, and $\{y_{k+\frac{1}{2}}\}$ and the edges  be such that $w_{k}$, $x_{k-\frac{1}{2}}$, $x_{k+\frac{1}{2}}$ is a complete graph on $3$ vertices, and so is $w_{k}$, $y_{k-\frac{1}{2}}$, $y_{k+\frac{1}{2}}$.  Then LF is obtained by replacing these complete $3$-graphs with copies of SG, see Figure~\ref{fig:LF}.

\begin{figure}
\includegraphics[width=.9\textwidth]{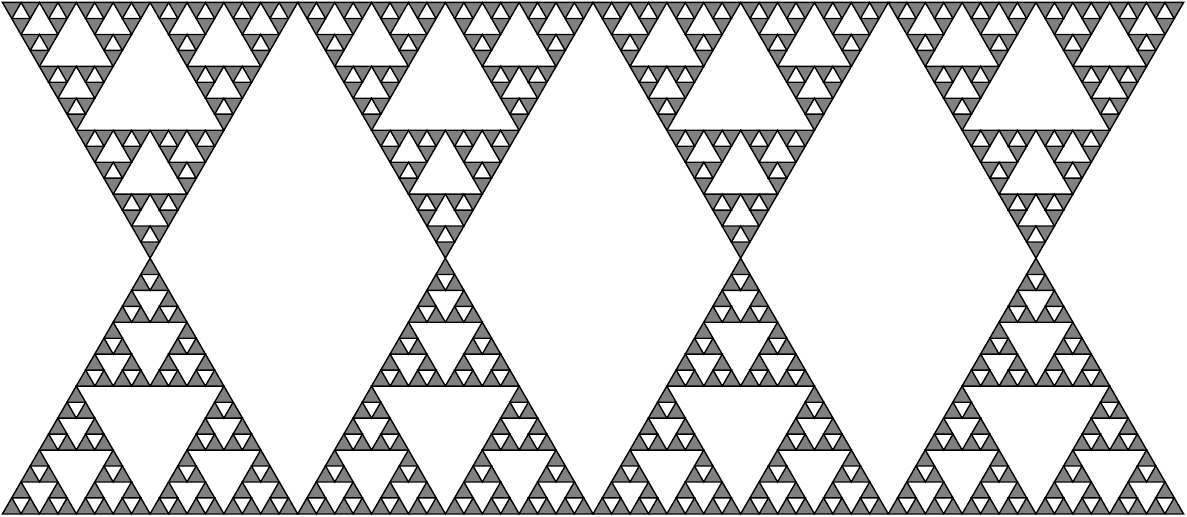}
\caption{The Ladder Fractafold}\label{fig:LF}
\end{figure}

According to the analysis in~\cite{ST12}, the spectrum of LF can be determined from the graph of the cells and their connectivity.  If we label the cell with vertices $w_{k}$, $x_{k-\frac{1}{2}}$, $x_{k+\frac{1}{2}}$ by $a_{k}$ and that with vertices $w_{k}$, $y_{k-\frac{1}{2}}$, $y_{k+\frac{1}{2}}$ by $b_{k}$ and treat $\{a_{k}\}\cup\{b_{k}\}$ as vertices of a graph $\Gamma$ with edges when the corresponding cells intersect, then $\Gamma$ is a ladder as shown in Figure~\ref{fig:GammaGraphs}.    If $-\Delta_{\Gamma}$ is the usual discrete Laplacian on $\Gamma$ it has absolutely continuous spectrum $[0,6]$.  One can prove the resolvent is unbounded by considering two sets of functions that satisfy an eigenfunction equation but are not in $L^{2}$: $\{\phi_{\theta}\}$  such that $\phi_{\theta}(a_{k})=\phi_{\theta}(b_{k})=e^{ik\theta}$  with eigenvalue $2-2\cos\theta$ (these are even in the reflection exchanging $a_{k}$ and $b_{k}$), and $\{\psi_{\theta}\}$  such that $\psi_{\theta}(a_{k})=-\psi_{\theta}(b_{k})=e^{ik\theta}$ with eigenvalue $4-2\cos\theta$ (these are odd in the reflection exchanging $a_{k}$ and $b_{k}$).  In both cases $0\leq\theta\leq\pi$.  From Theorem~3.1 of~\cite{ST12} and their discussion in Example~5.2, this spectrum is the same as that of $ -\Delta_{\Gamma_{0}}$.  Moreover they relate the spectrum $\sigma(-\Delta_{\Gamma_{0}})$ to that of the Laplacian on the fractafold as follows.

\begin{figure}
\includegraphics[width=\textwidth]{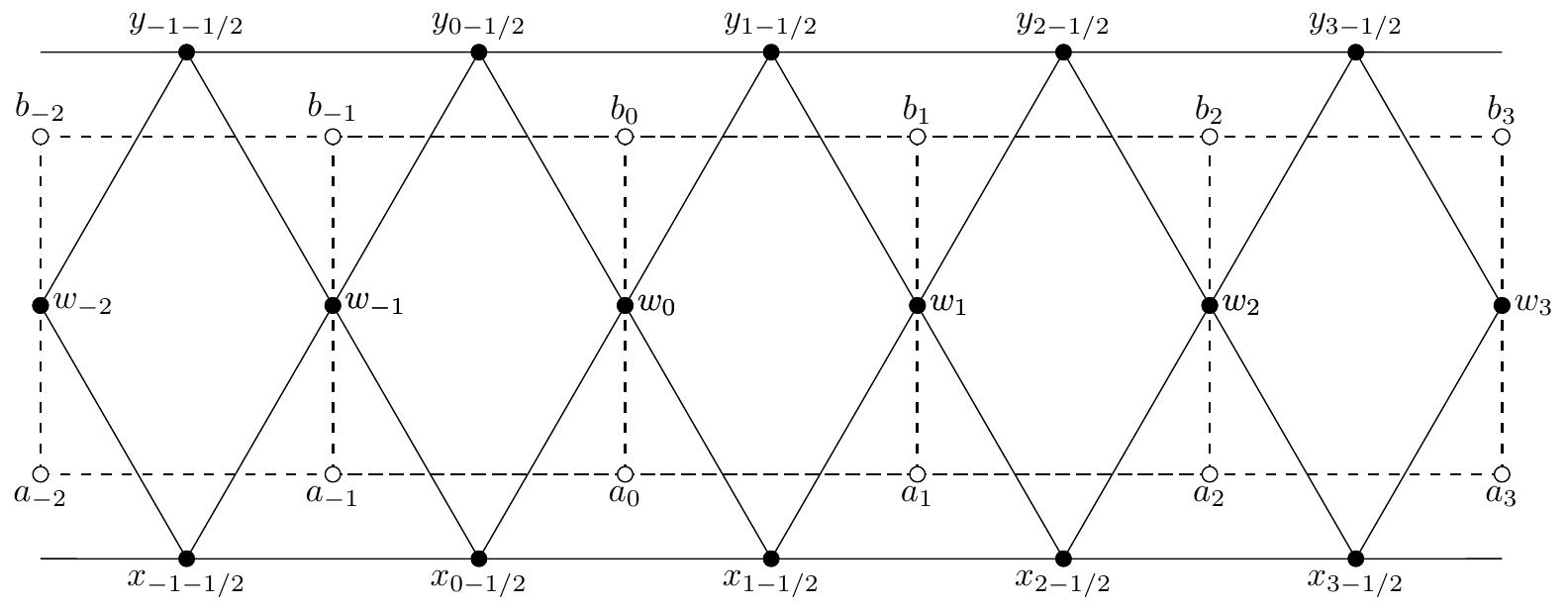}
\caption{The graphs $\Gamma_0$ (unfilled verteces and dashed edges), and $\Gamma$ (filled verteces, solid edges)}\label{fig:GammaGraphs}
\end{figure}

\begin{thm}[\protect{Theorem~2.3 of~\cite{ST12}}] \label{BigTheorem}
Using the function $\mathcal{R}$ from~(S\ref{S:fixation}) let
\begin{gather*}
	\Sigma_{\infty} = 5\biggl( \mathcal{R}\{2\} \cup \bigcup_{0}^{\infty} 5^{m}\mathcal{R}\{3,5\}\biggr),\\
	\Sigma'_{\infty} = 5\biggl( \bigcup_{m=0}^{\infty} 5^{m} \mathcal{R}\{3,5\}\biggr)\subset\Sigma_{\infty}.
	\end{gather*}
Then for $\Delta$ the Laplacian on the fractafold obtained by gluing according to $\Gamma_{0}$ 
\begin{equation*}
	 \mathcal{R}(\sigma (- \Delta_{\Gamma_0})) \cup \Sigma_{\infty}' \subset \sigma (- \Delta_{\text{LF}}) \subset\ \mathcal{R} (\sigma (- \Delta_{\Gamma_0})) \cup \Sigma_{\infty}.
	\end{equation*}
\end{thm}

To connect this to the study of the magnetic operator $\Mag^{\beta b}$ we ``fold'' the ladder along the center-line parallel to its length, so the point $x_{k+\frac{1}{2}}$ is identified with $y_{k+\frac{1}{2}}$ for all $k$, and obtain a fractafold is in Figure~\ref{fig:FLF}, which we call the folded ladder fractafold, or FLF.   The FLF is a covering space for SG in which the loop around the central hole of the $V_1$ graph has been trivialized. The covering map takes each cell $a_{k}$ in the fractafold to a $1$-cell of SG in a $3$-periodic manner, identifying $a_{k}$ with the cell $F_{k\,\text{mod}\,3}(SG)$, $w_{k}$ with $p_{k\,\text{mod}\,3}\in V_{0}$ and mapping both $x_{k+\frac{1}{2}}$ and $y_{k+\frac{1}{2}}$ to the same point of $V_{1}\setminus V_{0}$. We arrange the map so that the line through the $x_{k+\frac{1}{2}}$ wraps in a counterclockwise direction around the central hole in the $V_{1}$ graph as $k$ increases.

\begin{lem}
There is a bijection taking each Neumann eigenfunction $f$ of $\Delta$ on SG with eigenvalue $\lambda$ to a solution $\tilde{f}$ of $\Delta_{\text{LF}}\tilde{f}=\lambda\tilde{f}$ which is symmetrical under the central line reflection  and is $3$-periodic.
\end{lem}
\begin{proof}
For the definition and properties of the Laplacian on LF and FLF we refer to~\cite{ST12}.  A function  satisfying $\Delta_{\text{FLF}}\hat{f}=\lambda\hat{f}$ on  FLF unfolds to give a function $\tilde{f}$ on LF.  This function satisfies $\Delta_{\text{LF}}\tilde{f}=\lambda\tilde{f}$ if and only if its normal derivatives at each $w_{k}$ sum to zero; given the symmetry, this happens if and only if $d\hat{f}=0$ at all points $w_{k}$.  At the same time, the period $3$ covering of SG by FLF ensures that $3$-periodic solutions of $\Delta_{\text{FLF}}\hat{f}=\lambda\hat{f}$ on FLF correspond to eigenfunctions on SG in such a way that the normal derivatives at points $w_{k}$ correspond to those on $V_{0}$.
\end{proof}
\begin{rmk}
We could do something similar for the Dirichlet eigenfunctions on SG by considering antisymetry in the center line.
\end{rmk}

\begin{figure}
\includegraphics[width=\textwidth]{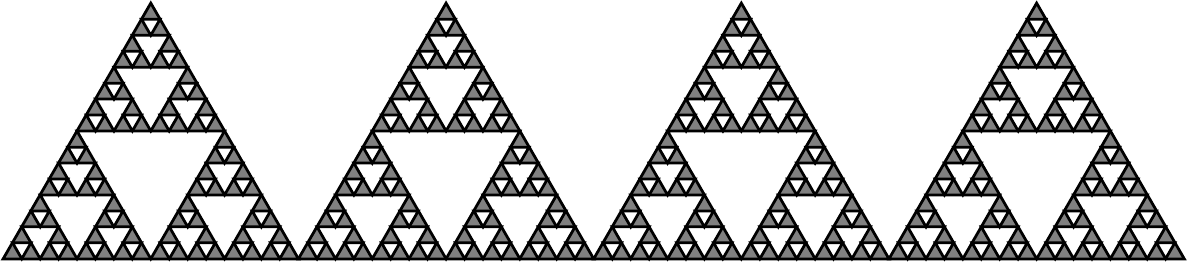}
\caption{The folded Sierpinski Ladder Fractafold}\label{fig:FLF}
\end{figure}

More importantly, the same thing happens for the magnetic operator $\Mag^{\beta b}$.  The only modification required for the proof is that the symmetric unfolding of a solution of  $\Delta_{\text{FLF}}\hat{f}=\lambda\hat{f}$ from FLF to LF gives a solution of $\Delta_{\text{LF}}\tilde{f}=\lambda\tilde{f}$  if and only if $d^{\beta b}\hat{f}(w_{k})$ sums to zero for all $k$.  However~\eqref{eqn:magnormalderivfromusual} and the fact that $db=0$ sums to zero at each $w_{k}$ ensures the Neumann condition is still the correct one. To make this argument we need $1$-forms and magnetic forms on LF and FLF; their properties are substantially similar to those on SG, and we refer to~\cite{IRT,HR} for more details. 
\begin{cor}
There is a bijective map which takes each Neumann eigenfunction of $\Mag^{\beta b}$ on SG  with eigenvalue $\lambda$ to a solution $\tilde{f}$ of $\Mag^{\beta b}_{\text{LF}}\tilde{f}=\lambda\tilde{f}$ that is symmetrical under the central line reflection  and $3$-periodic.  Here $\Mag^{\beta b}_{\text{LF}}$ is the magnetic operator corresponding to the symmetric $3$-periodic lift of $\beta b$ to LF.
\end{cor}

The preceding result is significant because passing to FLF trivializes the loop where $b$ is not exact, so we might expect $\beta b$ to be exact on FLF.   This is not literally true because the periodic extension of $\beta b$ to FLF will not have finite energy, simply because it is periodic.  However our reasoning regarding the gauge transformation is still valid: we can define $e^{i\beta B}$, which is globally continuous and locally in the domain of the Dirichlet form on FLF, such that $\Mag^{\beta b}f=e^{-i\beta B}\Delta_{\text{FLF}}(e^{i\beta B}f)$ for any $f$ in the domain of $\Mag^{\beta b}$ with compact support, and take limits to extend this operation to $L^{2}$.

\begin{thm}
The spectrum of the  Neumann magnetic operator $\Mag^{\beta b}$ on SG is
\begin{equation*}
\sigma (M^{\beta b}) = \mathcal{R} \Bigl\{ 2 - 2\cos\bigl( \frac{2k\pi}{3} - \frac{2\beta}{\sqrt{30}} \bigr)  \Bigr\}^{2}_{k=0} \cup \Sigma_{\infty}'
\end{equation*}
\end{thm}
\begin{proof}
The periodic extension of $\beta b$ to FLF has gauge $e^{iB}$ where $B$ is harmonic on each cell $a_{k}$ and has values 
\begin{equation*}
	B(x_{k+\frac{1}{2}})
	= \frac{2\beta k}{\sqrt{30}}+\frac{1}{2}
	\end{equation*}
and $B(w_{k})=0$ for all $k$. We use the same notation for the symmetric extension to LF.  The gauge transformation is valid and reduces the problem to finding those elements of the spectrum of the Laplacian on LF for which the associated function is symmetric in the center line and, after application of the gauge transformation, is $3$-periodic.  By Theorem~\ref{BigTheorem} and elementary arguments about the eigenfunctions associated to $\Sigma_{\infty}$ and $\Sigma'_{\infty}$ this includes all of $\Sigma'_{\infty}$ but not $5\mathcal{R}\{2\}$.  The remaining values correspond to spectral values from the  symmetric functions $\phi_{\theta}$ on $\Gamma$.  According to~\cite{ST12} the corresponding functions on LF are equal to $e^{i(k+\frac{1}{2})\theta}$ at $x_{k+\frac{1}{2}}$ and $e^{ik\theta}$ at $w_{k}$.  When multiplied by the gauge, these have
\begin{equation*}
	e^{iB}\phi_{\theta}(x_{k+\frac{1}{2}})
	=\exp i\Bigl( \bigl(k+\frac{1}{2}\bigr)\theta + \frac{2\beta k}{\sqrt{30}}+\frac{1}{2}\Bigr)
	\end{equation*}
which is periodic of period $3$ in $k$ if and only if $3\theta+\frac{6\beta}{\sqrt{30}}\equiv0\mod2\pi$.  Using this to determine $\theta$, the fact that the eigenvalue on $\Gamma$ was $2-2\cos\theta$ and the preceding reasoning from Theorem~\ref{BigTheorem} completes the proof.
\end{proof}

\bibliography{magneticformsbibliography}
\bibliographystyle{amsplain}


\end{document}